\documentclass[11pt,a4paper]{article}

\usepackage{amsmath, amssymb, amsthm} 
\usepackage{thmtools} 
\usepackage{xspace}
\usepackage{enumitem}
\usepackage{xcolor}
\usepackage{hyperref}
\usepackage{mathrsfs}
\usepackage{stackrel}
\usepackage{enumitem}

\setlist[description]{font=\normalfont}

\allowdisplaybreaks
\def\QED{$\blacksquare$}
\def\inQED{$\square$}

\renewenvironment{proof}
{\vspace{1ex}\noindent{\bf Proof.}\hspace{0.5em}}{\hfill \QED \vspace{1ex}}

\newenvironment{proofof}[1]
{\vspace{1ex}\noindent{\bf Proof of #1.}\hspace{0.5em}}{\hfill \QED \vspace{1ex}}

\allowdisplaybreaks

\theoremstyle{plain}
\newtheorem{theorem}{Theorem}[section]
\newtheorem{lemma}[theorem]{Lemma}

\newtheorem{proposition}[theorem]{Proposition}
\newtheorem{observation}[theorem]{Observation}

\newtheorem{remark}[theorem]{Remark}

\newcommand{\sm}{\setminus}
\DeclareMathOperator{\Ex}{\mathbb{E}}
\renewcommand{\Pr}{\mathbb{P}}

\makeatletter
\def\moverlay{\mathpalette\mov@rlay}
\def\mov@rlay#1#2{\leavevmode\vtop{%
   \baselineskip\z@skip \lineskiplimit-\maxdimen
   \ialign{\hfil$\m@th#1##$\hfil\cr#2\crcr}}}
\newcommand{\charfusion}[3][\mathord]{
    #1{\ifx#1\mathop\vphantom{#2}\fi
        \mathpalette\mov@rlay{#2\cr#3}
      }
    \ifx#1\mathop\expandafter\displaylimits\fi}
\makeatother

\newcommand{\discup}{\charfusion[\mathbin]{\cup}{\cdot}}
\newcommand{\bigdiscup}{\charfusion[\mathop]{\bigcup}{\cdot}}

\let\eps=\varepsilon
\let\theta=\vartheta
\let\rho=\varrho
\let\phi=\varphi

\usepackage{framed}

\makeatletter
\renewcommand*{\eqref}[1]{%
  \hyperref[{#1}]{\textup{\tagform@{\ref*{#1}}}}%
}
\makeatother

\title{Minors, connectivity, and diameter \\ in randomly perturbed sparse graphs}

\author{Elad Aigner-Horev \thanks{School of Computer Science, Ariel University, Ariel 40700, Israel. Email: {\tt horev@ariel.ac.il}.}
\quad Dan Hefetz \thanks{School of Computer Science, Ariel University, Ariel 40700, Israel. Email: {\tt danhe@ariel.ac.il}.}
\quad Michael Krivelevich \thanks{School of Mathematical Sciences, Tel Aviv University, Tel Aviv 6997801, Israel. 
Email: {\tt krivelev@tauex.tau.ac.il}. Research supported in part by USA--Israel BSF grant 2018267.}}

\begin{document}
\date{}
\maketitle

\begin{abstract}
Extremal properties of sparse graphs, randomly perturbed by the binomial random graph are considered. It is known that every $n$-vertex graph $G$ contains a complete minor of order $\Omega(n/\alpha(G))$. We prove that adding $\xi n$ random edges, where $\xi > 0$ is arbitrarily small yet fixed, to an $n$-vertex graph $G$ satisfying $\alpha(G) \leq \zeta(\xi)n$ asymptotically almost surely results in a graph containing a complete minor of order $\tilde \Omega \left( n/\sqrt{\alpha(G)}\right)$; this result is tight up to the implicit logarithmic terms. 

For complete topological minors, we prove that there exists a constant $C>0$ such that adding $C n$ random edges to a graph $G$ satisfying $\delta(G) = \omega(1)$, asymptotically almost surely results in a graph containing a complete topological minor of order $\tilde \Omega(\min\{\delta(G),\sqrt{n}\})$; this result is tight up to the implicit logarithmic terms. 

Finally, extending results of Bohman, Frieze, Krivelevich, and Martin for the dense case, we analyse the asymptotic behaviour of the vertex-connectivity and the diameter of randomly perturbed sparse graphs. 
\end{abstract}

\section{Introduction}

Given an $n$-vertex graph $G$ and a distribution $\mathcal{R}_n$ over all $n$-vertex graphs, the union $G \cup R$ with $R \sim \mathcal{R}_n$ defines a distribution over the $n$-vertex supergraphs of $G$, referred to as the {\em random perturbation} of $G$ (with respect to $\mathcal{R}_n$). More generally, given a family of $n$-vertex graphs $\mathcal{G}_n$, we say that $\mathcal{G}_n \cup \mathcal{R}_n$ asymptotically almost surely (a.a.s. hereafter) satisfies a given property $\mathcal{P}$, if $\lim_{n \to \infty} \Pr[G \cup R \in \mathcal{P}] = 1$, whenever $G \in \mathcal{G}_n$ and $R \sim \mathcal{R}_n$. We say that $\mathcal{G}_n \cup \mathcal{R}_n$ a.a.s. does not satisfy $\mathcal{P}$, if there exists a graph $G \in \mathcal{G}_n$ such that $\lim_{n \to \infty} \Pr[G \cup R \in \mathcal{P}] = 0$, whenever $R \sim \mathcal{R}_n$. Given any graph parameter $f(\cdot)$, we write $f(G \cup \mathbb{G}(n,p))$ to denote the random variable $f(G \cup R)$ where $R \sim \mathbb{G}(n,p)$. 

The study of randomly perturbed graphs dates back to the work of Bohman, Frieze, and Martin~\cite{BFM03} who proved that $\mathcal{G}_{n,\delta} \cup \mathbb{G}(n,\Theta_\delta(n^{-1}))$ is a.a.s. Hamiltonian, where $\delta > 0$ is a constant, $\mathcal{G}_{n,\delta}$ denotes the family of $n$-vertex graphs of minimum degree $\delta(G) \geq \delta n$, and $\mathbb{G}(n,p)$ denotes the binomial random graph with edge-probability $p$. Soon afterwards Krivelevich, Sudakov, and Tetali~\cite{KST} studied the Ramsey properties of $\mathcal{G}_{n,d} \cup \mathbb{G}(n,p)$, where $\mathcal{G}_{n,d}$ denotes the family of $n$-vertex graphs with edge density at least $d $, and $d > 0$ is independent of $n$. 

The two aforementioned results mark the initiation of the two most dominant strands of research pertaining to randomly perturbed (hyper)graphs; these being the study of the emergence of spanning configurations in such (hyper)graphs~~\cite{AHK,BTW17,BHKM18,BHKMPP18,BMPP18,DRRS18,DiazGeo, DiazPower, DiazReg,HZ18,KKS16,KKS17,MM18} and the study of their (anti-)Ramsey properties~\cite{ADHLlarge,ADHLsmall,AHhamilton,AHTrees,AHP,DKM21,DMT20,DT19,Powierski19}. Modulo few exceptions, in the results comprising these two strands, the (hyper)graphs being perturbed are dense and the random perturbation is binomial. 

\bigskip

In the present paper, we analyse the asymptotic behaviour of the order of the largest complete minor, the order of the largest complete topological minor, the vertex-connectivity, and the diameter of (possibly sparse) graphs that are randomly perturbed using the binomial random graph. All of our results, to be detailed next, are not far from being tight.

A graph $H$ is said to be a {\sl minor} of a graph $G$ if $H$ can be obtained from a subgraph of $G$ through a series of edge-contractions. The {\em Hadwiger number} of a graph $G$, denoted $h(G)$, is the largest integer $r$ such that $G$ contains a $K_r$-minor. Hadwiger's conjecture~\cite{Hadwiger} stipulates that $h(G) \geq \chi(G)$ holds for every graph $G$. Coupled with the trivial bound $\chi(G) \geq n/\alpha(G)$, Hadwiger's conjecture, if true, implies that $h(G) \geq \lceil \frac{n}{\alpha(G)} \rceil$. Conjecturing that $h(G) \geq \lceil \frac{n}{\alpha(G)} \rceil$ thus forms a natural relaxation of Hadwiger's conjecture. The weaker bound $h(G) \geq \left\lceil \frac{n}{2\alpha(G)-1} \right\rceil$ was proved by Duchet and Meyniel~\cite{DM82}. 
Following a rather lengthy series of improvements, the current state of the art is due to Balogh and Kostochka~\cite{BK11} who proved that $h(G) \geq \left\lceil\frac{n}{(2-c)\alpha(G)} \right\rceil$ holds where $c = (80 - \sqrt{5392})/126 > 0.052$ is some constant. 


Our first result asserts that adding $(1/2+o(1))n$ random edges to an $n$-vertex graph $G$ a.a.s. results in a graph $H$ for which $h(H) = \tilde \Omega\left(n/\sqrt{\alpha(G)}\right)$ holds.


\begin{theorem} \label{thm:main-minors}
For every $\varepsilon > 0$, there exists an $\alpha > 0$ such that
\begin{equation}\label{eq:minor-order}
h(G \cup \mathbb{G}(n,p)) = \Omega\left(\frac{n}{\sqrt{\log n}\cdot \max\left\{\sqrt{\alpha(G)},\; \log n\right\}}\right)
\end{equation}
holds a.a.s., whenever $p := p(n) = \frac{1 + \varepsilon}{n}$ and $G$ is an $n$-vertex graph satisfying $\alpha(G) \leq \alpha n$.
\end{theorem}





Theorem~\ref{thm:main-minors} is tight up to logarithmic terms. To see this, let $G$ be an $n$-vertex graph comprised of $n/(k+1)$ vertex-disjoint cliques, each of size $k+1$; note that $e(G) = nk/2$ and $\alpha(G)=n/(k+1)$. Let $R \sim \mathbb{G}(n,p)$, where $p := p(n) \leq \frac{k}{2n}$. Then, a.a.s. $e(R) \leq n k/2$ entailing $e(G \cup R) \leq nk/2 + n k/2 = nk$. Any graph $H$ that is a minor of $G \cup R$ must then satisfy $e(H) \leq e(G \cup R) \leq nk$. In particular, $h(G \cup R) = O(\sqrt{nk})  = O\left(n/\sqrt{\alpha(G)}\right)$. The logarithmic terms appearing in Theorem~\ref{thm:main-minors} are artefacts of our approach; it would be interesting to know whether these can be improved upon or removed.

Most vertices of an $n$-vertex graph $G$ are of degree $\Omega\left(n/\alpha(G) \right)$ (see Observation~\ref{obs:dev-avg-degree} in the next section). Hence, imposing an upper bound on the independence number of a graph yields a lower bound on its average degree. It is then reasonable to contemplate whether replacing upper bounds on $\alpha(G)$ with corresponding lower bounds on the average degree of $G$ in Theorem~\ref{thm:main-minors} would yield similar results. The following example demonstrates that this is far from the truth. Let $G = (A \discup B, E)$ be a complete bipartite graph, where $k := |A| = \sqrt{n}$ and $|B| = n-k$; note that $\delta(G) = k$ and $h(G) = k+1$. For $h(G \cup \mathbb{G}(n,p)) = \omega(k)$ to hold a.a.s., $h(\mathbb{G}(n,p)[B]) = \omega(k)$ must hold a.a.s. as well. Indeed, at most $k$ branch sets may intersect $A$ and any two branch sets that are fully contained in $B$ may only be connected via a vertex of $A$ or an edge of $\mathbb{G}(n,p)$ . Having $h(\mathbb{G}(n,p)[B]) = \omega(k)$ requires the random perturbation to be of size $\omega\left(k^2\right) = \omega(n)$; this, however, is a.a.s. not the case for $p = O(1/n)$. To summarise, adding $\Omega(n)$ random edges to an $n$-vertex graph with minimum degree $\sqrt{n}$ may have insignificant impact on its Hadwiger number; on the other hand, by Theorem~\ref{thm:main-minors}, adding $(1/2 + o(1)) n$ random edges to any $n$-vertex graph with independence number $\sqrt{n}$ (and thus average degree $\Omega(\sqrt{n})$) a.a.s. yields a graph with Hadwiger number $\Omega \left(\frac{n^{3/4}}{\sqrt{\log n}} \right)$. 

Given a constant $C >1$, it is known by a result in~\cite{FKS08}, that $h(\mathbb{G}(n,p)) = \Theta(\sqrt{n})$ holds a.a.s. whenever $p:=p(n) = C/n$. Consequently, 
a complete minor of order $\Theta(\sqrt{n})$ typically resides in the perturbed graph considered in Theorem~\ref{thm:main-minors} by virtue of the random edges alone. Therefore, for Theorem~\ref{thm:main-minors} to be informative, the right hand side of~\eqref{eq:minor-order} must dominate $\sqrt{n}$. This imposes that the stronger bound $\alpha(G) = O(n/\log n)$ be upheld for the theorem to be meaningful. Our second result addresses this issue. More importantly, it allows for sparser perturbations at the expense of obtaining a smaller (yet, still non-trivial) complete minor.


\begin{theorem}\label{thm:main-minors-sparse}
There exists a constant $c > 0$ such that given an integer $k := k(n) = o(n)$, the bound
$$
h(G \cup \mathbb{G}(n,p)) = \Omega\left(\frac{n}{\sqrt{\log n}\cdot \max\left\{\sqrt{\alpha(G)},\; \log n\right\}\cdot k}\right)
$$
holds a.a.s. whenever $G$ is an $n$-vertex graph satisfying $\alpha(G) \leq \frac{c n}{k}$ and $p:= p(n) = \frac{8}{nk}$. 
\end{theorem}




\bigskip

Hadwiger's conjecture has attracted a lot of attention and significant progress towards its resolution has been made over the years. A classical result by Kostochka~\cite{Kos84} and by Thomason~\cite{Tho95} asserts that {\sl every} graph $G$ satisfies $h(G) = \Omega\left(\frac{\overline{d}(G)}{\sqrt{\log (\overline{d}(G))}} \right)$, where $\bar d(G)$ denotes the average degree of $G$; a short proof of this theorem can be found in~\cite{AKS22}. Building on the work of Norin, Postle, and Song~\cite{NPS19,Postle} (and references therein), a major breakthrough was recently obtained by Delcourt and Postle~\cite{DP22} who proved that any graph $G$ with chromatic number $r$ satisfies $h(G) =\Omega\left(r/\log \log r \right)$. 


For the binomial random graph $\mathbb{G}(n,p)$, Hadwiger's conjecture is known to be true a.a.s. essentially throughout the whole range of $p := p(n)$(see~\cite{FKS08,KN21} and references therein). 
In fact, these results assert that almost all graphs $G$ satisfy
$$
h(G) = \Omega\left(\sqrt{\frac{e(G)}{\log (\overline{d}(G))}}\right) \;\; \;  \text{and} \; \; \; \chi(G) = O\left(\frac{\overline{d}(G)}{\log (\overline{d}(G))} \right),
$$
implying that Hadwiger's conjecture holds true for almost all graphs in a rather strong sense. 

Fountoulakis, K\"uhn and Osthus~\cite{FKS09} proved that the random $3$-regular graph $\mathbb{G}_{n,3}$ a.a.s. satisfies 
$h(\mathbb{G}_{n,3}) = \Theta(\sqrt{n})$. Their proof exploits the fact that $\mathbb{G}_{n,3}$ can be generated by first sampling a Hamilton cycle $C$ over $[n]$ uniformly at random and then randomly perturbing $C$ using a perfect matching, sampled uniformly at random from all perfect matchings over $V(C)$. As far as we know, this is the first result regarding the Hadwiger number of randomly perturbed graphs. 

More recently, Kang, Kang, Kim, and Oum~\cite{KKKO20} proved that for every $\omega(n^{-2}) = p := p(n) \leq 2/n$ there exists a constant $C>0$ such that given an $n$-vertex connected graph $G$ of maximum degree $\Delta \leq pn^2/C$, a.a.s. 
\begin{equation*} 
h(G \cup \mathbb{G}(n,p)) = \Omega\left(\min\left\{\sqrt{\frac{pn^2}{\log \Delta}}, \frac{pn^2}{\Delta \sqrt{\log \Delta}} \right\} \right)
\end{equation*}
holds. 

In light of the aforementioned results, it is natural to wonder whether perturbed variants of Hadwiger's conjecture can be proved, by which we mean proving that $h(G \cup \mathbb{G}(n,p)) \geq \chi(G \cup \mathbb{G}(n,p))$ holds a.a.s. for any sufficiently large $n$-vertex graph $G$ and non-trivial $p:=p(n)$. A graph $G$ is said to be \emph{vertex-transitive} if for any pair of vertices $(u, v) \in V(G) \times V(G)$, there is an automorphism of $G$ which maps $u$ to $v$. Our third result is a  consequence of Theorem~\ref{thm:main-minors}; it asserts that a rather modest random perturbation is sufficient in order to a.a.s. generate a graph satisfying Hadwiger's conjecture in a fairly strong sense provided that the graph being perturbed is vertex-transitive. 

\begin{proposition} \label{th::HadwigerVT}
Let $8 < k := k(n)$ be an integer. Then, there exist constants $c_1, c_2 > 0$ such that $G \cup \mathbb{G}(n,p)$ a.a.s. satisfies Hadwiger's conjecture,  whenever $G$ is a vertex-transitive graph on $n$ vertices satisfying $c_1 k^2 \log^3 n \leq \alpha(G) \leq \frac{c_2 n}{k}$, and $p :=p(n) = \frac{8}{nk}$.
\end{proposition}

\begin{remark}
{\em For a significant range of admissible values of $k$, the perturbed graph $G \cup \mathbb{G}(n,p)$, considered in Proposition~\ref{th::HadwigerVT}, satisfies $h(G \cup \mathbb{G}(n,p)) = \omega(\chi(G \cup \mathbb{G}(n,p)))$ asymptotically almost surely.}
\end{remark}


%

A graph $H$ is said to be a {\em subdivision} of a graph $K$ if $H$ can be obtained from $K$ by repeatedly subdividing\footnote{The act of subdividing an edge $e$ in a graph entails the removal of $e$ and the addition of a degree two vertex whose sole neighbours are the ends of $e$.} edges.  A graph $K$ is said to be a {\em topological minor} of a graph $G$ if $G$ contains a subdivision of $K$ as a subgraph. For a graph $G$, let $\mathrm{tcl}(G)$ denote the largest integer $r$ such that $G$ contains a subdivision of $K_r$. 

The order of magnitude of $\mathrm{tcl}(\mathbb{G}(n,p))$ is known provided that $p := p(n) \geq C/n$, where $C>1$ is independent of $n$~\cite{AKS79,BC81}. Roughly stated, these results collectively assert that $\mathrm{tcl}(\mathbb{G}(n,p)) = \Theta(\min\{pn,\sqrt{n}\})$ holds a.a.s., whenever $p$ is as above; the first term arises due to the triviality $\mathrm{tcl}(G) \leq \Delta(G) + 1$ holding for any graph $G$; the second term accounts for the so-called space limitation that obstructs the accommodation of large complete topological minors~\cite{BC81,EF81}. Our fourth result reads as follows. 

\begin{theorem}\label{thm:tcl}
There exists a constant $C>0$ such that $\mathrm{tcl}(G \; \cup \; \mathbb{G}(n,p)) \geq  \min \left\{ \delta(G)/8, \sqrt{n/60 \log_2 n}\right\}$ holds a.a.s. whenever $p:=p(n) \geq C/n$ and $G$ is an $n$-vertex graph satisfying $\delta(G)  = \omega(1)$.
\end{theorem}

Up to logarithmic terms, Theorem~\ref{thm:tcl} is tight with respect to both lower bounds proclaimed in its statement. Indeed, let $k := k(n)$ and $p := p(n)$ be such that $\frac{2\log n}{n} \leq p = o(k/n)$. Let $G$ be a $k$-regular graph on $n$ vertices and let $R \sim \mathbb{G}(n,p)$. Then a.a.s. $\Delta(R) = o(k)$ implying that $\mathrm{tcl}(G \cup R) \leq \Delta(G \cup R) +1 = (1+o(1))k$. This establishes the tightness of Theorem~\ref{thm:tcl} with respect to the minimum degree of $G$. As for the second bound, let $G \sim \mathbb{G}(n,1/3)$ and $R \sim \mathbb{G}(n,o(1))$. Then, $G \cup R \sim \mathbb{G}(n, p)$ for some $1/3 \leq p \leq 1/2$. It then follows by a classical result of Erd\H{o}s and Fajtlowicz~\cite{EF81} (see~\cite{BC81} as well) that $\mathrm{tcl}(G \cup R) = O(\sqrt{n})$.

\bigskip

Our next result provides estimates for the vertex-connectivity of randomly perturbed graphs. It is well-known that $p = (\log n + (k-1) \log \log n)/n$ is the threshold for $\mathbb{G}(n,p)$ being $k$-vertex-connected (and also for being $k$-edge-connected and for having minimum degree $k$). 


While there are disconnected $n$-vertex graphs $G$ satisfying $\delta(G) = \Omega(n)$, Bohman, Frieze, Krivelevich, and Martin~\cite{BFKM04} proved that adding $\omega(1)$ random edges to such a graph $G$ a.a.s. results in a graph $H$ satisfying $\kappa(H) \geq k$, provided that $k$ is independent of $n$. For $\omega(1) = k  = O(n)$, they proved that $\Omega(k)$ random edges suffice in order to a.a.s. obtain a $k$-connected graph (consult~\cite{BFKM04} regarding the implicit constants in the Big O and Big $\Omega$ notation); their results are tight in terms of the number of random edges needed. Our fifth result provides an extension of these results by accommodating the random perturbation of sparse graphs. 

\begin{theorem} \label{th::VertexCon}
Let $n$, $k := k(n)$ and $s := s(n)$ be positive integers satisfying $k \leq s/17$. 
\begin{itemize}
\item [$(a)$] Let $G = (V,E)$ be an $n$-vertex graph satisfying $\delta(G) = s$, where $s = o(n)$. Then, a.a.s. $\kappa(G \cup {\mathbb G}(n,p)) \geq k$ holds, whenever $p := p(n) \geq \frac{c (k + \log(n/s))}{n s}$ and $c$ is a sufficiently large constant. 

\item [$(b)$] There exists an $n$-vertex graph $G_0$ satisfying $\delta(G_0) \geq s$ such that $\kappa(G_0 \cup H) < k$ holds, whenever $H$ is a graph having fewer  than $k \lfloor n/(s+1) \rfloor /2$ edges. Moreover, if $\omega(1) = s = o(n)$, then $\kappa(G_0 \cup \mathbb{G}(n,p)) = 0$ holds a.a.s. whenever $p := p(n) \leq \frac{(1 - o(1)) \log(n/s)}{n s}$.  
\end{itemize}
\end{theorem}

\bigskip
Lastly, we consider the asymptotic behaviour of the {\sl diameter} of randomly perturbed sparse graphs. For a connected graph $G$, let $\mathrm{diam}(G) = \max \left \{\textrm{dist}_G(u,v) : u,v \in \binom{V(G)}{2} \right\}$ denote the {\em diameter} of $G$, where $\textrm{dist}_G(u,v)$ is the length of a shortest $uv$-path in $G$; if the latter is disconnected, then $\mathrm{diam}(G) = \infty$. 
The threshold for the property $\mathrm{diam}(\mathbb{G}(n,p)) = d$ is known for every fixed integer $d \geq 2$; it is roughly $\tilde \Omega(n^{-\frac{d-1}{d}})$ (see, e.g.~\cite[Theorem~7.1]{FK}). If $p := p(n) = \frac{\omega(\log n)}{n}$, then $\mathrm{diam}(\mathbb{G}(n,p)) = (1+o(1)) \frac{\log n}{\log (np)} $ a.a.s.~\cite[Theorem~7.2]{FK}.

While there are $n$-vertex graphs $G$ satisfying $\delta(G) = \Omega(n)$ whose diameter is infinite, Bohman, Frieze, Krivelevich, and Martin~\cite{BFKM04} proved that such graphs are {\sl not far} from having constant diameter. More precisely, they proved that adding  $\omega(1)$ random edges to such a graph $G$ a.a.s. results in a graph whose diameter is at most 5, that adding $\Omega(\log n)$ edges to such a graph a.a.s. results in a graph whose diameter is at most 3, and that adding $\Omega(n \log n)$ random edges a.a.s. has the resulting graph having diameter at most 2.
 Their results are tight in terms of the number of random edges needed. Our sixth result extends these results by allowing the graph being perturbed to be sparse. 

\begin{theorem} \label{th::Diameter} 
Let $n$ and $k := k(n) = o(n)$ be integers, let $p := p(n) = \omega \left(\frac{\log (n/k)}{n k} \right)$, and let $q = 1 - (1-p)^{k^2}$.
\begin{itemize}
\item [$(a)$] Let $G$ be an $n$-vertex graph with minimum degree $k$. Then, a.a.s. $\emph{diam}(G \cup \mathbb{G}(n,p)) \leq \frac{(5 + o(1)) \log (n/k)}{\log (n q/(2k))}$.

\item [$(b)$] There exists an $n$-vertex graph $G_0$ satisfying $\delta(G_0) \geq k$ such that a.a.s. $\emph{diam}(G_0 \cup \mathbb{G}(n,p)) > \frac{(1 - o(1)) \log (n/k)}{\log (n q/k)}$. 
\end{itemize}
\end{theorem}

The rest of this paper is organised as follows. In Section~\ref{sec:minors} we prove Theorem~\ref{thm:main-minors}. In Section~\ref{sec:minors-sparse-perturb} we prove Theorem~\ref{thm:main-minors-sparse}. In Section~\ref{sec:vx-trans} we prove Proposition~\ref{th::HadwigerVT}. In Section~\ref{sec::topMinor} we prove Theorem~\ref{thm:tcl}. In Section~\ref{sec::VerCon} we prove Theorem~\ref{th::VertexCon}. Finally, in Section~\ref{sec::diameter} we prove Theorem~\ref{th::Diameter}. 
Throughout this paper, unless explicitly stated otherwise, $\log$ stands for the natural logarithm. We omit floor and ceiling
signs whenever they are not crucial.


\section{Complete minors in randomly perturbed graphs}\label{sec:minors}

In this section, we prove Theorem~\ref{thm:main-minors}. We commence with the following observation alluded to in the introduction. Given a graph $G$ and a positive integer $k$, let $B := \{v \in V(G) : \deg_G(v) < k\}$. Then,
$$
\alpha(G) \geq \alpha(G[B]) \geq |B|/(\Delta(G[B])+1) \geq |B|/k,
$$
implying that $|B| \leq k\alpha(G)$. We have thus shown the following.

\begin{observation}\label{obs:dev-avg-degree}
For every $\eta > 0$ and every $n$-vertex graph $G$, all but at most $\eta n$ vertices $v \in V(G)$ satisfy $\deg_G(v) \geq \lfloor \eta \frac{n}{\alpha(G)} \rfloor$. 
\end{observation}

Given an $n$-vertex graph $G$, a set $X \subseteq V(G)$ of size $\ell$ formed by choosing each of its members independently and uniformly at random from $V(G)$ without replacement, is said to be an $\ell$-{\em uniform subset of $G$} ($\ell$-set, hereafter). A $(k,\ell)$-{\em ensemble} of $G$ is a sequence $(X_1,\ldots,X_k)$ of $\ell$-sets, all sampled in succession such that the ambient set from which $X_i$ is sampled is $V(G) \sm \bigcup_{j=1}^{i-1} X_j$. The {\em size} of such an ensemble is $\left| \bigcup_{i=1}^k X_i \right| = k \ell$. For distinct indices $i,j \in [k]$, set $\mathcal{X}^{(i,j)} = \emptyset$ if $j < i$, and $\mathcal{X}^{(i,j)} = \bigcup_{r=i}^j X_r$ otherwise; put $\mathcal{X} := \mathcal{X}^{(1,k)}$. 


\medskip

The following lemma forms our main step towards identifying the potential branch sets of the sought after minor. 


\begin{lemma}\label{lem:ensemble}
Given $\gamma \in (0,1/12)$, set $\alpha$ such that $\gamma^2 \alpha^{-1} \geq 2$. Let $G$ be an $n$-vertex graph satisfying $\alpha(G) \leq \alpha n$ and define the quantities
\begin{equation}\label{eq:kl}
\ell := \max\left\{\frac{2}{\gamma} \sqrt{\alpha(G)},\;\frac{32}{\gamma} \log n\right\} \quad \text{and} \quad k := \gamma \frac{n}{\ell}.
\end{equation}
Let $(X_1,\ldots,X_k)$ be a $(k,\ell)$-ensemble of $G$ and let $(U_1,\ldots,U_k)$ be a $(k,\ell)$-ensemble of $G \setminus \mathcal{X}$. Then, the following properties hold a.a.s. simultaneously.
\begin{description}	
	\item [(E.1)] For every $i \in [k]$, there exists a set $N_i \subseteq N_G(X_i)$ satisfying $|N_i| \geq k/3$ and $N_i \cap \mathcal{X} = \emptyset$;
		
	\item [(E.2)] $|\{j \in [k] : N_i \cap U_j \neq \emptyset\}| = \Omega_\gamma(k)$ for every $i \in [k]$. 
	
\end{description}
\end{lemma}


\begin{proof}
The proof proceeds via two rounds of exposure. First, the ensemble $(X_1,\ldots,X_k)$ is shown to satisfy Property~(E.1) asymptotically almost surely. Then, an ensemble $(X_1, \ldots, X_k)$ satisfying Property~(E.1) is fixed and subsequently the ensemble $(U_1,\ldots, U_k)$ is sampled and proven to a.a.s.~satisfy Property~(E.2). 

\medskip

Starting with Property~(E.1), choose the $\ell$-sets $X_1, \ldots, X_k$ in succession. 
Fix $i \in [k]$ and consider the ongoing formation of $X_i$. 
For $j \in [\ell]$, let $X^{(j)}_i$ denote the subset of $X_i$ formed after precisely $j$ vertices have been placed in $X_i$. For every $i \in [k]$ let $A_i^{(0)} \subseteq  A_i^{(1)}\subseteq  \ldots \subseteq A_i^{(\ell)}$ be a sequence of sets constructed as follows. First, set $A_i^{(0)} = \emptyset$. Assume we have already defined $A_i^{(0)} \subseteq \ldots \subseteq A_i^{(j)}$ for some $0 \leq j < \ell$ and now wish to define $A_i^{(j+1)}$. If $\left| A_i^{(j)} \right| \geq k/2$, set $A_i^{(j+1)} = A_i^{(j)}$; otherwise, define $A_i^{(j+1)}$ as follows. Set 
\begin{equation}\label{eq:Rj}
R_{j+1} = V(G) \sm \left( \mathcal{X}^{(1,i-1)} \cup X^{(j)}_i \cup A_i^{(j)}\right),
\end{equation}
and define $R'_{j+1} = R_{j+1} \sm B_{j+1}$, where 
$$
B_{j+1} := \left\{v \in R_{j+1}: \deg_{G[R_{j+1}]}(v) < \gamma \frac{|R_{j+1}|}{\alpha(G[R_{j+1}])} \right\}.
$$
Let $v_{j+1} \in X_i^{(j+1)} \setminus X_i^{(j)}$ denote the $(j+1)$st vertex sampled and added to $X_i$. If $v_{j+1}$ is chosen from $R'_{j+1}$, then set $A_i^{(j+1)} = A_i^{(j)} \cup N_{G[R_{j+1}]}(v_{j+1})$; otherwise, set $A_i^{(j+1)} = A_i^{(j)}$. 
For every $i \in [k]$ let $A_i = A_i^{(\ell)}$ and let $N_i = A_i \sm \mathcal{X}$; note that $N_i \subseteq N_G(X_i)$ and $N_i \cap \mathcal{X} = \emptyset$. It thus remains to prove that a.a.s. $|N_i| \geq k/3$ for every $i \in [k]$. For every $i \in [k]$, define the events 
$$
\mathcal{E}_i^{(1)} : |A_i| \geq k/2, \quad \text{and} \quad \mathcal{E}_i^{(2)} : |A_i \cap \mathcal{X}| \leq 2\gamma |A_i|.
$$
Using the assumption $\gamma < 1/12$, it follows that 
\begin{equation}\label{eq:E2}
\Pr[\text{Property~(E.1) fails}] \leq \Pr\left[\;\mathcal{E}_i^{(1)} \; \text{fails for some $i \in [k]$}\;\right]  + \sum_{i=1}^k \Pr\left[\mathcal{E}_i^{(2)} \; \text{fails} \; \Big| \; \mathcal{E}_i^{(1)} \; \text{holds}\right].
\end{equation}
We prove that the two terms appearing on the right hand side of~\eqref{eq:E2} are both $o(1)$.

Commencing with the first term, fix some $i \in [k]$ and some integer $0 \leq j < \ell$ for which $\left|A_i^{(j)} \right| < k/2$. Since $\left|\mathcal{X}^{(1,i-1)} \cup X^{(j)}_i\right| \leq k \ell \overset{\eqref{eq:kl}}{=} \gamma n$, $\left|A_i^{(j)} \right| < k/2 < \gamma n$, and $\gamma < 1/3$, it follows that $|R_{j+1}| \geq (1 - 2 \gamma) n \geq \gamma n$.
Observation~\ref{obs:dev-avg-degree}, applied to $G[R_{j+1}]$, implies that $|B_{j+1}| \leq \gamma \cdot |R_{j+1}|$. Therefore 
\begin{equation}\label{eq:R'}
|R'_{j+1}| \geq |R_{j+1}| (1 - \gamma) \geq \gamma n/2,
\end{equation}
where the last inequality holds for $\gamma \leq 1/2$. If $v_{j+1}$ is chosen from $R'_{j+1}$, then its addition results in $A_i^{(j+1)} = A_i^{(j)} \cup N_{G[R_{j+1}]}(v_{j+1})$ being set; the latter leads to an increase of at least 
$$
\left\lfloor \gamma \frac{|R_{j+1}|}{\alpha(G[R_{j+1}])} \right\rfloor \geq \left\lfloor\gamma^2 \frac{n}{\alpha(G)}\right\rfloor \geq \lfloor \gamma^2 \alpha^{-1} \rfloor \geq 1
$$ 
in the size of the eventual set $A_i$, where for the last inequality we rely on our choice of $\alpha$. This choice also supports the inequality
$$
\left\lfloor \gamma^2 \frac{n}{\alpha(G)} \right\rfloor  \cdot \left\lceil \frac{\gamma}{4}\ell \right\rceil \overset{\eqref{eq:kl}}{\geq} \gamma \frac{n}{2\ell} = k/2.
$$ 
Hence, if $v_j$ is chosen from $R'_j$ for at least $\gamma \ell/4$ indices $j \in [\ell]$, then $|A_i| \geq k/2$.

For $i \in [k]$ and $r \in [\ell]$, the addition of the $r$th vertex $v_r$ to $X_i$ is termed {\sl successful} if $\left|A^{(r-1)} \right| \geq k/2$ or $v_r \in R'_r$. If the former holds, then the addition is successful with probability one; otherwise, by~\eqref{eq:R'}, the addition is successful with probability at least $\gamma/2$. These probability bounds hold regardless of the outcome of previous additions to $X_i$. Let $Z_i$ denote the number of successful vertex-additions to $X_i$ and let $Z\sim \mathrm{Bin}(\ell, \gamma/2)$. Then, $\Pr[Z_i < h] \leq \Pr[Z < h]$ holds for every $h$.
An application of Chernoff's inequality (see, e.g.~\cite[Theorem~2.1]{JLR}) and the union-bound then yield
\begin{align} \label{eq:E_1}
\Pr\left[\mathcal{E}_i^{(1)}\; \text{fails for some $i \in [k]$}\right] &\leq \Pr\left [Z_i < \frac{\gamma}{4}\ell\; \text{for some $i \in [k]$}\right] \leq k \cdot \exp\left\{- \frac{\gamma}{16}\ell \right\} \nonumber \\ 
&\overset{\eqref{eq:kl}}{\leq} \exp \left\{ \log k - \frac{\gamma}{16} \cdot \frac{32}{\gamma} \log n \right\} = o(1),
\end{align}
where the last equality holds since $k \leq n$. 


\medskip
To handle the second term appearing on the right hand side of~\eqref{eq:E2}, it suffices to prove that  
$$
\Pr\left[|A_i \cap \mathcal{X}| > 2 \gamma |A_i| \; \Big| \; |A_i| \geq k/2 \right] = o(1/k) 
$$
holds for every $i \in [k]$. Fix any $i \in [k]$. Note that $|\mathcal{X}| = k \ell = \gamma n$ and that whenever a vertex is sampled and added to $\mathcal{X}$, the probability that it is in $A_i$ as well is at most $\frac{|A_i|}{(1 - \gamma)n}$ (in some cases this probability is actually zero) and this holds regardless of the outcome of previous additions to $\mathcal{X}$. Indeed, $\mathcal{X}^{(1,i)} \cap A_i = \emptyset$ by definition (see~\eqref{eq:Rj}) and when elements of $\mathcal{X}^{(i+1,k)}$ are sampled, $A_i$ is already determined. This stronger bound is not needed for our calculations.  It follows that $\Pr \left[ \left|A_i \cap \mathcal{X} \right| > h \, \Big| \, |A_i| \geq k/2 \right] \leq \Pr[L > h]$ for every $h$, where $L \sim \mathrm{Bin}\left(\gamma n, \frac{|A_i|}{(1-\gamma)n} \right)$. Since $\gamma \leq 1/3$, it follows that $\Ex[L] = \frac{\gamma}{1-\gamma}|A_i| \leq 1.5 \gamma |A_i|$. An application of Chernoff's inequality then yields
\begin{equation} \label{eq:E_2}
\Pr\left[ \left|A_i \cap \mathcal{X} \right| > 2\gamma |A_i| \; \Big| \; |A_i| \geq k/2 \right] \leq \exp \{- \Omega_\gamma(|A_i|) \}  \overset{|A_i| \geq k/2}{=} \exp \{-\Omega_\gamma(k) \} = o(1/k),
\end{equation}
where the last equality holds since $k = \Omega(\sqrt{n})$ by~\eqref{eq:kl}.



Combining~\eqref{eq:E2}, \eqref{eq:E_1}, and~\eqref{eq:E_2}, implies that Property~(E.1) holds a.a.s. as required. 

\bigskip

We proceed to establishing Property~(E.2). Fix a $(k,\ell)$-ensemble $(X_1,\ldots, X_k)$ satisfying Property~(E.1). For $i \in [k]$, let $N_i \subseteq N_G(X_i) \setminus \mathcal{X}$ be a set of size $|N_i| = k/3$. For $i \in [k]$ and an ensemble $(U_1, \ldots, U_k)$ let
$u_i = |\{j \in [k]: U_j \cap N_i \neq \emptyset\}|$. It suffices to prove that there exists a constant $c > 0$ such that 
\begin{equation*} 
\Pr\left[u_i < c k \, \Big| \, |N_i| = k/3 \right] = o(1/k)
\end{equation*}
holds for every $i \in [k]$. Hence, fix some $i \in [k]$ for the remainder of the proof. Generate the sets $U_1, \ldots, U_k$ one by one. Assume that, for some $j \in [k]$, we have already sampled $U_1, \ldots, U_{j-1}$ and now wish to sample $U_i$. Let $N'_i = N_i \cap (U_1 \cup \ldots \cup U_{j-1})$ and let $N''_i = N_i \setminus N'_i$. Let $Z \sim \mathrm{HG}((1 - 2 \gamma) n, k/3, \gamma n)$; in particular $\mathbb{E}(Z) = \frac{\gamma k}{3 (1 - 2 \gamma)} \leq \gamma k/2$, where the inequality holds since $\gamma < 1/12$. Since $|\mathcal{X}| = \gamma n$ and $|U_1 \cup \ldots \cup U_{j-1}| = (j-1) \ell \leq k \ell = \gamma n$, it follows that $\Pr[|N'_i| \geq h \, | \, |N_i| = k/3] \leq \Pr[Z \geq h]$ holds for every $h$. An application of Chernoff's inequality (see, e.g.~\cite[Theorem~2.10]{JLR}) thus implies that
\begin{align*}
\Pr\left[|N''_i| \leq k/4 \; \Big| \; |N_i| = k/3\right] \leq \Pr\left[|N'_i| \geq \gamma k \; \Big| \; |N_i| = k/3\right] \leq \Pr[Z \geq 2 \mathbb{E}(Z)] \leq \exp (- c' k), 
\end{align*}
where the first inequality holds since $\gamma < 1/12$ and $c' := c'(\gamma) > 0$ is an appropriate constant. Therefore
\begin{align*}
\Pr\left[N_i \cap U_j = \emptyset \; \Big| \; |N''_i| \geq k/4\right] = \frac{\binom{n - |\mathcal{X} \cup U_1 \cup \ldots \cup U_{j-1} \cup N''_i|}{\ell}}{\binom{n - |\mathcal{X} \cup U_1 \cup \ldots \cup U_{j-1}|}{\ell}} \leq 1 - c'',
\end{align*}
where $c'' > 0$ is an appropriate constant (the latter inequality holds by the birthday paradox and can also be verified via a direct calculation). Therefore
$$
\Pr\left[N_i \cap U_j = \emptyset \, \Big| \, |N_i| = k/3 \right] \leq \Pr\left[|N''_i| \leq k/4 \, \Big| \, |N_i| = k/3 \right] + \Pr\left[N_i \cap U_j = \emptyset \, \Big| \, |N''_i| \geq k/4 \right] \leq 1 - c''/2.
$$
We conclude that $\Pr\left[N_i \cap U_j \neq \emptyset \, | \, |N_i| = k/3 \right] \geq c''/2$ for every $j \in [k]$, and this holds regardless of whether $N_i \cap \left(\bigcup_{r \in J} U_r \right)$ is empty or not for any $J \subseteq [k] \setminus \{j\}$. Let $X \sim \textrm{Bin}(k, c''/2)$; then $\Pr[u_i < h \, | \, |N_i| = k/3] \leq \Pr[X < h]$ holds for every $h$. It thus follows by Chernoff's inequality that
$$
\Pr\left[u_i < c'' k/4 \, \Big| \, |N_i| = k/3 \right] \leq \Pr\left[X < \mathbb{E}(X)/2 \right] \leq \exp (- c'' k/16) = o(1/k)
$$ 
as required.
\end{proof}

\bigskip

We are now ready to prove Theorem~\ref{thm:main-minors}.

\begin{proofof}{Theorem~\ref{thm:main-minors}}
Given $\varepsilon > 0$, set an auxiliary constant
\begin{equation} \label{eq::zeta}
0 < \zeta < \min\left\{\frac{1}{12},\frac{\varepsilon^2}{48}\right\}.
\end{equation}
Let $G$ be an $n$-vertex graph, with $n$ being sufficiently large, and set $\ell$ and $k$ to be as in~\eqref{eq:kl} with $\gamma = \zeta$.

Extending a result of Ajtai, Koml\'os, and Szemer\'edi~\cite{AKS81}, Krivelevich and Sudakov~\cite{KS13} (see also~\cite[Page~220]{AlonSpencer}) proved that $R \sim \mathbb{G}\left(n, (1 + \varepsilon)/n \right)$ a.a.s.~contains a path $P$ of length $\left\lfloor \frac{\varepsilon^2}{12}n \right\rfloor$. Let $P_1, \ldots, P_{2k}$ be a collection of $2 k $ vertex-disjoint sub-paths of $P$, each of length $\ell$. Such a partition exists since $2 k (\ell+1) \leq |V(P)|$ holds by~\eqref{eq:kl} and~\eqref{eq::zeta}. 

Any subgraph $K$ of $R \sim \mathbb{G}(n,p)$ is distributed uniformly over all copies of $K$ in $K_n$. To see this, note that $R\sim \mathbb{G}(n,p)$ can be generated as follows. First, generate $R' \sim \mathbb{G}(n,p)$. Then, pick a permutation $\pi \in S_n$ uniformly at random and set $R:=([n],\{\pi(u)\pi(v): uv \in E(R')\})$. The resulting distribution coincides with that of $\mathbb{G}(n,p)$. It follows that 
$(V(P_1),\ldots,V(P_k))$ and $(V(P_{k+1}),\ldots,V(P_{2k}))$ both form $(k,\ell)$-ensembles of $G$ with the {\sl added} property that the members across both ensembles are pairwise disjoint.   

Lemma~\ref{lem:ensemble}, applied to $(V(P_1),\ldots,V(P_k))$ and $(V(P_{k+1}),\ldots,V(P_{2k}))$ with $\gamma = \zeta$, asserts that a.a.s.~for every $i \in [k]$ there exists a set $N_i \subseteq N_G(V(P_i))$ satisfying Property~(E.1) and such that the sets $(N_1,\ldots,N_k)$ and $(V(P_{k+1}),\ldots,V(P_{2k}))$ satisfy Property~(E.2). It follows that $G \cup R$ a.a.s.~contains some graph $H$ as a minor such that $v(H) = 2k$, $e(H) = \Omega(k^2)$, and whose branch sets are $V(P_1),\ldots, V(P_{2k})$. The Kostochka-Thomason Theorem~\cite{AKS22,Kos84,Tho95} then asserts that $H$, and thus also $G \cup R$, contains a complete minor of order 
$$
\Omega\left(\frac{k}{\sqrt{\log k}}\right) = \Omega \left(\frac{n}{\sqrt{\log n} \cdot  \ell}\right)
$$
as required. 
\end{proofof}

\section{Complete minors arising from sparse perturbations}\label{sec:minors-sparse-perturb}

In this section, we deduce Theorem~\ref{thm:main-minors-sparse} from Theorem~\ref{thm:main-minors}. 

\begin{proofof}{Theorem~\ref{thm:main-minors-sparse}}
Let $\eps \in (0,1)$ be fixed and set $\alpha := \alpha(\eps)$ to be the constant guaranteed by Theorem~\ref{thm:main-minors}; note that $0 < \alpha \leq 1$ . Set $c = \alpha/3$. Since $k = o(n)$, it follows that $p = o(k^{-2})$. Let $r = n/(2k)$; we claim that $G$ admits pairwise vertex-disjoint paths $P_1, \ldots, P_r$, each of length $k - 1$. Indeed, assuming that for some $i \in [r]$, the paths $P_1, \ldots, P_{i-1}$ have already been constructed, the path $P_i$ is found as follows. Let $G_i = G \setminus (V(P_1) \cup \ldots \cup V(P_{i-1}))$, so that $|V(G_i)| = n - (i-1) k \geq n - r k = n/2$. Let $G'_i$ be the graph obtained from $G_i$ by repeatedly discarding vertices of degree at most $k - 1$. Note that $\alpha(G'_i) \leq \alpha(G)$. Moreover, $G'_i$ is non-empty and thus $\delta(G'_i) \geq k$ holds by construction. Indeed, let $S = V(G_i) \setminus V(G'_i)$ and assume for a contradiction that $|S| \geq n/2$. Observe that $G_i[S]$ is $(k-1)$-degenerate, implying that 
$$
\alpha(G) \geq \alpha(G_i[S]) \geq \frac{|S|}{k} \geq \frac{n}{2k} > \frac{c n}{k} \geq \alpha(G),
$$ 
which is clearly a contradiction. It follows that $G'_i$ contains a path of length $k - 1$ which we denote by $P_i$.

Define two auxiliary graphs, namely $\Gamma_G$ and $\Gamma_R$, such that $V(\Gamma_G) = V(\Gamma_R) = \{u_1, \ldots, u_r\}$ and, for any $1 \leq i < j \leq r$, there is an edge of $\Gamma_G$ (respectively, $\Gamma_R$) connecting $u_i$ and $u_j$ if and only if $E_G(V(P_i), V(P_j)) \neq \emptyset$ (respectively,  $E_R(V(P_i), V(P_j)) \neq \emptyset$). Then  

$$
\alpha(\Gamma_G) \leq \alpha(G)\leq \frac{c n}{k} \leq \alpha r
$$ 
and $\Gamma_R \sim \mathbb{G}(r, q)$, where $q \geq \frac{1+\eps}{r}$. To justify the latter inequality, note that 
$p k^2 = o(1)$ by assumption, implying that
\begin{align*}
\Pr[\{u_i, u_j\} \notin E(\Gamma_R)] &= (1-p)^{k^2} \leq \exp\left(- p k^2 \right) = 1- pk^2 + o(pk^2) \\
& \leq 1 - \frac{pk^2}{2} = 1 - \frac{4 k}{n} \leq 1 - \frac{1+\eps}{r}
\end{align*}
holds for any $1 \leq i < j \leq r$, where the last equality holds since $p = \frac{8}{nk}$ and the last inequality holds since $r = \frac{n}{2k}$ and $\eps < 1$. 

Since $r = \omega(1)$ holds by our assumption that $k = o(n)$, it follows by Theorem~\ref{thm:main-minors} that 
\begin{align*}
h(G \cup R) &\geq h(\Gamma_G \cup \Gamma_R) = \Omega\left(\frac{r}{\sqrt{\log r} \cdot \max\left\{\sqrt{\alpha(\Gamma_G)}, \log r \right\}} \right)  \\
&= \Omega \left(\frac{n}{\sqrt{\log n} \cdot \max\left\{\sqrt{\alpha(G)},\log n\right\} \cdot k} \right).
\end{align*} 
holds asymptotically almost surely. 
\end{proofof}

\newpage

\section{Hadwiger's conjecture for randomly perturbed vertex-transitive graphs}\label{sec:vx-trans}. 



In this section, we prove Proposition~\ref{th::HadwigerVT}. The following result is an immediate consequence of Proposition 1.3.4 and Lemma~1.6.4 from~\cite{SUbook}.

\begin{proposition} \label{prop::VTlog}
Let $G$ be a vertex-transitive graph on $n$ vertices. Then, $\chi(G) \leq (1 + \log n) n/\alpha(G)$.
\end{proposition}

\begin{proofof}{Proposition~\ref{th::HadwigerVT}}
Let $k, p$ and $G$ be as in the premise of the proposition, and let $c_2 > 0$ be the constant whose existence is ensured by Theorem~\ref{thm:main-minors-sparse}. Since $p = \frac{8}{nk}$ and $k > 8$, the results seen in~\cite[Chapter~2]{FK} assert that $\chi(R) \leq 3$ a.a.s. holds as at this edge-probability all components of the random graph are a.a.s. either trees or unicyclic. Consequently, 
$\chi(G \cup R) \leq \chi(G) \chi(R) \leq 3 \chi(G)$ 
holds asymptotically almost surely. On the other hand, Theorem~\ref{thm:main-minors-sparse}, Proposition~\ref{prop::VTlog}, and the lower bound imposed on $\alpha(G)$, where $c_1$ is assumed to be a sufficiently large constant, together imply that a.a.s.
$$
h(G \cup R) = \Omega \left(\frac{n}{\sqrt{\alpha(G) \log n} \cdot k}\right) \geq  3\chi(G).
$$
\end{proofof}

\section{Complete topological minors in randomly perturbed graphs} \label{sec::topMinor}

In this section, we prove Theorem~\ref{thm:tcl}. The core property of the random perturbation facilitating this proof is stated and established in Section~\ref{sec:rnd-prop}. A proof of Theorem~\ref{thm:tcl} can be seen in Section~\ref{sec:tcl-proof}. 

\subsection{Properties of random graphs}\label{sec:rnd-prop}

The main result of this section reads as follows. 

\begin{proposition}\label{prop:diam}
There exists a constant $C >0$ such that $R \sim \mathbb{G}(n,C/n)$ a.a.s. satisfies the property that every subset $U \subseteq [n]$ of size $|U| \geq 0.9n$ contains a subset $U' \subseteq U$ of size $|U'| \geq 0.8n$ such that $\mathrm{diam}(R[U']) \leq 3 \log_2 n$.
\end{proposition}


Before proving Proposition~\ref{prop:diam}, we collect some properties of random graphs facilitating our proof.

%
%
%
%




\begin{lemma}\label{lem:min-deg}
There exists a constant $C > 0$ such that $R \sim \mathbb{G}(n,C/n)$ a.a.s. satisfies the property that for every $U \subseteq [n]$ of size $|U| \geq 0.9 n$, there exists a subset $U' \subseteq U$ of size $|U'| \geq 0.8 n$ such that $\delta(R[U']) \geq C/2$. 
\end{lemma}

\begin{proof}
A standard application of Chernoff's inequality shows that the probability that there are disjoint sets $X$ and $W$ such that $|X| = 0.1 n$, $|W| \geq 0.8 n$, and $e_R(X, W) < 0.05 C n$ is at most
$$
2^n \cdot 2^n \cdot \Pr \left[\textrm{Bin} \left(0.08 n^2, C/n \right) \leq 0.05 C n \right] < 4^n e^{- \Theta(C) n} = o(1),
$$
where the last equality holds for a sufficiently large constant $C$. Assume then, for the remainder of the proof, that $e_R(X, W) \geq 0.05 C n$ whenever $X$ and $W$ are disjoint sets of sizes $|X| = 0.1 n$ and $|W| \geq 0.8 n$.


Fix some $U \subseteq [n]$ of size $|U| \geq 0.9 n$. Repeatedly remove vertices of $U$ whose degree in the current subgraph of $R[U]$ is strictly smaller than $C/2$; denote the resulting subset of $U$ by $U'$. Suppose for a contradiction that $|U'| < 0.8 n$. Let $X$ be an arbitrary subset of $U \setminus U'$ of size $0.1 n$ and let $W = U \setminus X$; note that $|W| \geq 0.8 n$. Therefore
$$
0.05 C n \leq e_R(X, W) < |X| \cdot C/2 = 0.05 C n.
$$
This contradiction concludes the proof of the lemma as $\delta(R[U']) \geq C/2$ holds by construction.
\end{proof}



The next two lemmas consider the edge distribution of random graphs. 


\begin{lemma}\label{lem:span}
There exists a constant $C > 0$ such that $R \sim \mathbb{G}(n,C/n)$ a.a.s. satisfies the property that $e_R(X) < C|X|/8$ whenever $X \subseteq [n]$ is of size $|X| \leq n/e^{10}$. 
\end{lemma}

\begin{proof}
The probability that there exists a set $X \subseteq [n]$ of size $|X| \leq n/e^{10}$ such that $e_R(X) \geq C|X|/8$ is at most   
\begin{align*}
\sum_{t =1}^{n/e^{10}}\binom{n}{t} \binom{t^2}{Ct/8}\left(\frac{C}{n}\right)^{Ct/8} & \leq \sum_{t=1}^{n/e^{10}} \left(\frac{en}{t} \cdot \left(\frac{8 e t}{C} \right)^{C/8} \cdot \left( \frac{C}{n}\right)^{C/8}\right)^t \leq \sum_{t=1}^{n/e^{10}} \left(e^C \left(\frac{t}{n}\right)^{C/8-1} \right)^t  \\
& \leq \sum_{t=1}^{\log n} \left(e^C \left(\frac{t}{n}\right)^{C/9} \right)^t + \sum_{t=\log n}^{n/e^{10}} \left(e^C \left(\frac{t}{n}\right)^{C/9} \right)^t \\
& \leq \log n \cdot \frac{\log n}{n} + n \left(e^{C - 10 \cdot C/9}\right)^{\log n} = o(1) + n e^{-C \log n/9} = o(1),
\end{align*}
where throughout we rely on $C$ being sufficiently large. 
\end{proof}




\begin{lemma}\label{lem:disc}
There exists a constant $C > 0$ such that $R \sim \mathbb{G}(n,C/n)$ a.a.s. satisfies the property that $e_R(X,Y) > 0$ whenever $X,Y \subseteq [n]$ are disjoint subsets of sizes $|X|, |Y| \geq n/(2e^{10})$. 
\end{lemma}

\begin{proof}
The probability that there exist two disjoint subsets $X,Y \subseteq [n]$ of sizes $|X|, |Y| \geq n/(2e^{10})$ such that $E_R(X,Y) = \emptyset$ is at most
$$
2^n \cdot 2^n \cdot (1 - C/n)^{n^2/(4 e^{20})} \leq 4^n \cdot \exp \left\{- C n/(4 e^{20}) \right\} = o(1), 
$$
where the last equality holds for a sufficiently large constant $C$.
\end{proof}

We are ready to prove Proposition~\ref{prop:diam}. 

\begin{proofof}{Proposition~\ref{prop:diam}}
Let $C$ be a sufficiently large constant and let $R \sim \mathbb{G}(n,C/n)$; for the remainder of the proof we assume that $R$ satisfies the properties stated in Lemmas~\ref{lem:min-deg},~\ref{lem:span}, and~\ref{lem:disc}. Fix $U \subseteq [n]$ of size $|U| \geq 0.9 n$. Then, there exists a subset $U' \subseteq U$ of size $|U'| \geq 0.8 n$ such that $\delta(R[U']) \geq C/2$; denote $R[U']$ by $R'$. Given $u \in U'$ and a non-negative integer $i$, let $B_u^{(i)}$ denote the set of vertices in $U'$ whose distance from $u$ in $R'$ is at most $i$. Fix some $u \in U'$ and some non-negative integer $k$ for which $\left|B_u^{(k)}\right| \leq n/(2 e^{10})$. Denote $B_u^{(k)}$ by $X$ and let $Y = N_{R'}(X) = B_u^{(k+1)} \setminus B_u^{(k)}$. Suppose for a contradiction that $|Y| \leq |X|$. Then $|X \cup Y| \leq n/e^{10}$ and 
$$
e_{R}(X \cup Y) \geq e_{R'}(X \cup Y) \geq |X| \delta(R')/2 \geq C |X|/4 \geq C |X \cup Y|/8,
$$ 
contrary to the assertion of Lemma~\ref{lem:span}. It follows that $\left|B_u^{(k+1)} \right| = |X \cup Y| \geq 2|X| = 2 \left|B_u^{(k)} \right|$. Therefore, $\left|B_u^{(i)} \right| \geq \min \left\{2^i, n/(2 e^{10}) \right\}$ holds for any $u \in U'$ and any non-negative integer $i$. We conclude that $B_u^{(\log_2 n)} \cap B_v^{(\log_2 n)} \neq \emptyset$ or $e_{R'} \left(B_u^{(\log_2 n)},B_v^{(\log_2 n)} \right) > 0$ holds for any pair of distinct vertices $u,v \in U'$, giving rise to a $uv$-path in $R'$ of length at most $2 \log_2 n +1 \leq 3 \log_2 n$ between any two such vertices. 
\end{proofof}

\subsection{Proof of Theorem~\ref{thm:tcl}}\label{sec:tcl-proof}

Prior to proving Theorem~\ref{thm:tcl}, we state and prove the following lemma which is used in order to identify the branch vertices of the topological minor proclaimed to a.a.s. exist by this theorem.

\begin{lemma}\label{lem:stars}
Let $s$ and $n$ be positive integers satisfying $64 \leq s^2 \leq n$ and let $G$ be an $n$-vertex graph of minimum degree $\delta(G) \geq s$. Then, $G$ contains a family of $s/8$ pairwise vertex-disjoint copies of  $K_{1,s/4}$. 
\end{lemma}

\begin{proof}
Let $B = (X \discup Y, E)$ be a spanning bipartite subgraph of $G$ of minimum degree $\delta(B) \geq s/2$; any maximum cut of $G$ defines such a bipartite subgraph. Without loss of generality, assume that $|X| \leq n/2$. Let $S_1, \ldots, S_t$ be a maximal collection of pairwise vertex-disjoint copies of $K_{1, s/4}$ in $B$, each having its centre vertex\footnote{For $n \geq 2$, the unique vertex of $K_{1,n}$ whose degree exceeds one is called its {\em centre vertex}.} residing in $X$. Let $M$  denote the set of all centre vertices of said copies and let $L$ denote the set of all leaves of these copies.  

Assume for a contradiction that $t < s/8$ which in turn implies that $|L| = ts/4 < s^2/32$. Each vertex in $Y \sm L$ has at least $s/2 - t > s/2 - s/8 = 3s/8$ neighbours in $X \sm M$. In particular, 
$$
e_B(X \sm M, Y \sm L) \geq |Y \sm L| \cdot 3s/8 \geq (n/2 - s^2/32) \cdot 3s/8 \overset{s^2 \leq n}{\geq} ns/8
$$
holds, implying that there exists a vertex $x \in X \sm M$ satisfying 
$$
\deg_{B'}(x) \geq \frac{ns/8}{n/2} = s/4,
$$
where $B' := B[X \sm M, Y \sm L]$. 
This contradicts the maximality of the collection $S_1,\ldots,S_t$. 
\end{proof}

We are ready to prove Theorem~\ref{thm:tcl}. 

\begin{proofof}{Theorem~\ref{thm:tcl}}
Let $C >0$ be a constant as in the statement of Proposition~\ref{prop:diam}. Set 
$$
\ell = \min \left\{ \delta(G) / 8,  \sqrt{n/60\log_2 n} \right\},
$$
and let $S_1, \ldots, S_\ell$ be a collection of vertex-disjoint copies of $K_{1,2\ell}$ in $G$ ({\sl stars}, hereafter); such a collection exists in $G$ by Lemma~\ref{lem:stars}. Let $M$ denote the set of centre vertices of these stars; in what follows we provide a probabilistic construction of a topological $K_\ell$-minor  whose branch vertices coincide with $M$. 

For every $i \in [\ell]$ let $L_i$ denote the leaf vertices of $S_i$; set $L = \bigdiscup_{i \in [\ell]} L_i$, and fix an arbitrary linear ordering of the members of $L_i$ for every $i \in [\ell]$. Let $Z = V(G) \sm (M \cup L)$ and note that $|Z| \geq 0.95 n$ as $|M| + |L| = \ell + 2\ell^2 \leq n/\log_2 n = o(n)$. 
Expose the perturbation $R \sim \mathbb{G}(n,C/n)$ over $Z$ only and consider the following construction. Set $k=1$ and initialise $Z_k := Z$ and $L_i^{(k)} := L_i$ for every $i \in [\ell]$. Fix an arbitrary linear ordering of the elements of $\mathcal{P} := \{(i,j) : 1 \leq i < j \leq \ell\}$. While $k \leq \binom{\ell}{2}$, perform the following steps.
\begin{enumerate}[labelwidth=1.5cm,labelindent=10pt,leftmargin=2.2cm,label=\bfseries Step \arabic*.,align=left]
	\item [Step 1.] Let $U_k \subseteq Z_k$ be a subset 
	of size $|U_k| \geq 0.8n$ and such that $\mathrm{diam}
	((G \cup R)[U_k]) \leq 3 \log_2n$; the existence of $U_k$ is 
	justified below.
	
	\item [Step 2.] Let $i < j$ denote the elements of the $k$th pair of $\mathcal{P}$. 
	Scan the members of  $L_i^{(k)}$ according to the assumed
	ordering. Given a vertex $y \in L_i^{(k)}$ considered 
	throughout the scan, expose the edges of $R$ incident to 
	both $y$ and $U_k$. If no edges are revealed, then declare 
	$y$ to be a {\sl failure} (with respect to $S_i$) and proceed to the 
	next vertex 
	in the ordering. Halt upon reaching the first member of 
	$L_i^{(k)}$ that is not a failure (if no such vertex exists, then terminate); 
	denote this vertex by $y_i$. 
	Remove all failed vertices encountered throughout the scan as well as $y_i$ 
	from $L_i^{(k)}$, and let $L_i^{(k+1)}$ denote the 
	resulting set. 
		
	\item [Step 3.] Perform Step~2 over $L_j^{(k)}$; 
	let $y_j$ and $L_j^{(k+1)}$ denote the counterparts of 
	$y_i$ and $L_i^{(k+1)}$ defined in Step~2, respectively 
	(unless termination is reached). 
		
	\item [Step 4.] Connect $y_i$ and $y_j$ via a path 
	$P_{ij} \subseteq G \cup R$ such that $V(P_{ij}) \setminus \{y_i, y_j\} \subseteq U_k$
	and its length is at most $3 \log_2 n$. 
	
	\item [Step 5.] Set $Z_{k+1} := Z_k \sm V(P_{ij})$, increase $k$ by one and return to Step 1. 
\end{enumerate}

It remains to prove that the above probabilistic construction a.a.s. constructs $\binom{\ell}{2}$ paths $P_{ij}$, as defined in Step~4; in order to prove this, it suffices to prove that a.a.s. for every $i \in [\ell]$ there are at most $\ell$ failures with respect to $S_i$. We start by arguing that a.a.s. the set $U_k$, defined in Step~1, exists in each execution of Step~1. To see this, note that if $\binom{\ell}{2}$ paths are ever defined throughout the construction, then at most 
$$
\frac{n}{60 \log_2 n} \cdot 3\log_2 n = 0.05 n
$$ 
vertices are ever removed from $Z$. Consequently, $|Z_k| \geq 0.9 n$ holds throughout the process, allowing for an appeal to Proposition~\ref{prop:diam} in every execution of Step~1.

The probability that a vertex is declared a failure with respect to a given star $S_i$ is at most $(1 - C/n)^{0.8 n} \leq e^{- 0.8 C} \leq 1/4$, where the last inequality is owing to $C$ being sufficiently large. The number of failures seen for a given star $S_i$ throughout is then stochastically dominated by a random variable that is binomially distributed with parameters $2\ell$ and $1/4$. An application of Chernoff's inequality (see e.g.~\cite[Theorem~2.1]{JLR}) then yields
$$
\Pr[\text{Strictly more than $\ell$ failures occur for a given star}] \leq e^{-\Omega(\ell)}.
$$
A union bound over all stars then implies that
$$
\Pr[\text{There exists a star for which strictly more than $\ell$ failures occur}] \leq \ell e^{ - \Omega(\ell)} = o(1),
$$
where the last equality holds since $\ell = \omega(1)$ which is implied by $\delta(G) = \omega(1)$. 
\end{proofof}

\section{Vertex-connectivity of randomly perturbed graphs} \label{sec::VerCon}

In this section, we prove Theorem~\ref{th::VertexCon}. 
Our proof makes use of the following known result due to Mader.

\begin{theorem} [\cite{Mader}] \label{th::Mader}
Every graph of average degree at least $s$ admits an $s/4$-connected subgraph.
\end{theorem}
 
\begin{proofof}{Theorem~\ref{th::VertexCon}}
Starting with (a), let $\{V_1, \ldots, V_t\}$ be a family of pairwise-disjoint subsets of $V$ such that the subgraph $G[V_i]$ is $s/16$-connected for every $i \in [t]$, and $t$ is maximal for any family with this property (note that $t \geq 1$ holds by Theorem~\ref{th::Mader}). Observe that, in particular, $|V_i| \geq s/16$ for every $i \in [t]$ and thus $t \leq 16 n/s$. Let $W = V \setminus (V_1 \cup \ldots \cup V_t)$ and let $r = |W|$. We claim that $G[W]$ is $s/4$-degenerate. Indeed, suppose for a contradiction that there exists a set $Z \subseteq W$ such that the minimum degree in $G[Z]$ is at least $s/4$. Hence, by Theorem~\ref{th::Mader}, there exists a set $V_{t+1} \subseteq Z$ such that $G[V_{t+1}]$ is $s/16$-connected, contrary to the maximality of $t$. Let $(w_1, \ldots, w_r)$ be a degeneracy ordering of the elements of $W$, that is, $\deg_G \left(w_i, \bigcup_{j=1}^t V_j \cup \{w_1, \ldots, w_{i-1}\} \right) \geq 3s/4$ holds for every $i \in [r]$.

Given a non-trivial partition $I_1 \cup I_2$ of $[t]$, define sequences of sets $A_0 \subseteq A_1 \subseteq \ldots \subseteq A_r$ and $B_0 \subseteq B_1 \subseteq \ldots \subseteq B_r$, such that $A_i \cup B_i$ is a partition of $\bigcup_{j=1}^t V_j \cup \{w_1, \ldots, w_i\}$ for every $i \in [r]$, as follows. Set $A_0 = \bigcup_{i \in I_1} V_i$ and $B_0 = \bigcup_{i \in I_2} V_i$. Suppose we have already defined $A_0, \ldots, A_{i-1}$ and $B_0, \ldots, B_{i-1}$ for some $i \in [r]$, and now wish to define $A_i$ and $B_i$. If $\deg_G(w_i, A_{i-1}) \geq 3s/8$, then set $A_i = A_{i-1} \cup \{w_i\}$ and $B_i = B_{i-1}$, otherwise set $A_i = A_{i-1}$ and $B_i = B_{i-1} \cup \{w_i\}$. Observe that, given the partition $[t] = I_1 \cup I_2$, the sets $A_r$ and $B_r$ are uniquely determined. A graph $H$ with the same vertex-set as $G$ is said to satisfy the property ${\mathcal M}_k$ if for every non-trivial partition $[t] = I_1 \cup I_2$ there is a matching in $H$ of size $k$ such that each of its edges has one endpoint in $A_r$ and the other in $B_r$.

Let $R \sim {\mathbb G}(n,p)$, where $p := p(n) \geq \frac{c (k + \log(n/s))}{n s}$. Fix an arbitrary non-trivial partition $[t] = I_1 \cup I_2$, and let $A_r \cup B_r$ be the corresponding partition of $V$. If $R$ does not admit a matching of size $k$ between $A_r$ and $B_r$, then it admits an inclusion maximal matching $M$ between $A_r$ and $B_r$ of size $i$ for some $0 \leq i \leq k-1$; by maximality $E_R(A_r \setminus V(M), B_r \setminus V(M)) = \emptyset$. Assuming first that $k \geq \log(n/s)$, the probability of the latter event occurring is at most 
\begin{align*}
\sum_{i=0}^{k-1} \binom{|A_r| |B_r|}{i} p^i (1-p)^{(|A_r| - i) (|B_r| - i)} &\leq (1-p)^{|A_r| |B_r|} + \sum_{i=1}^{k-1} \left(\frac{e |A_r| |B_r| p}{i} \right)^i e^{- p (|A_r| - i) (|B_r| - i)} \\
&\leq e^{- p |A_r| |B_r|} + \sum_{i=1}^{k-1} \exp \{i \log (e |A_r| |B_r| p/i) - p (|A_r| - i) (|B_r| - i)\} \\
&\leq k \exp \left \{ k \log (e |A_r| |B_r| p/k) - p |A_r| |B_r|/300 \right\} \\
&\leq k \exp \left \{ k \left[\log \left(\frac{2 c e |A_r| |B_r|}{n s} \right) - \frac{c |A_r| |B_r|}{300 n s} \right] \right\} \\ 
&\leq k \exp \left \{ - \frac{c k |A_r| |B_r|}{600 n s} \right\} \leq k e^{- c' k |I_1|} \leq e^{- c'' |I_1| \log(n/s)},
\end{align*} 
where in the third inequality we use the fact that $f(i) := i \log(a/i)$ is increasing for $1 \leq i \leq \min \{k, a/3\}$, the fact that $\min \{|A_r|, |B_r|\} \geq s/16$, and our assumption that $k \leq s/17$; the fourth inequality holds by our assumption that $k \geq \log(n/s)$; the fifth inequality holds for a sufficiently large constant $c$ since $|A_r| |B_r| \geq n s/32$; in the sixth inequality we assume without loss of generality that $|A_r| \leq |B_r|$; additionally $c', c''$ are sufficiently large constants depending on $c$. 

Similarly, if $k < \log(n/s)$, then
\begin{align*}
\sum_{i=0}^{k-1} \binom{|A_r| |B_r|}{i} p^i (1-p)^{(|A_r| - i) (|B_r| - i)} &\leq k \exp \left \{ k \log (e |A_r| |B_r| p/k) - p |A_r| |B_r|/300 \right\} \\
&\leq k \exp \left \{ k \left[\log \left(\frac{2 c e |A_r| |B_r| \log(n/s)}{k n s} \right) - \frac{c |A_r| |B_r| \log(n/s)}{300 k n s} \right] \right\} \\ 
&\leq k \exp \left \{ - \frac{c |A_r| |B_r| \log(n/s)}{600 n s} \right\} \leq e^{- c'' |I_1| \log(n/s)},
\end{align*}

Either way, a union bound over all choices of $\emptyset \neq I_1 \subsetneq [t]$ then implies that the probability that $R$ does not satisfy ${\mathcal M}_k$ is at most
\begin{align*}
\sum_{i=1}^{t-1} \binom{t}{i} e^{- c'' i \log(n/s)} \leq \sum_{i=1}^{\infty} e^{i [\log(16 n/s) - c'' \log(n/s)]}  = o(1),
\end{align*}
where in the inequality above we use the fact that $t \leq 16n/s$, and in the equality above we use the fact that $c''$ is sufficiently large and the assumption $s = o(n)$.

It thus suffices to prove that if $R$ satisfies ${\mathcal M}_k$, then $G \cup R$ is $k$-connected. Assume then that $R$ satisfies ${\mathcal M}_k$ and suppose for a contradiction that there exists a set $S \subseteq V$ of size $k-1$ such that $(G \cup R) \setminus S$ is disconnected. Let $S$ be such a set and let $A \cup B$ be a non-trivial partition of $V \setminus S$ such that $E_{G \cup R}(A, B) = \emptyset$. Since $\kappa((G \cup R)[V_i]) \geq \kappa(G[V_i]) \geq s/16 \geq k$ for every $i \in [t]$, it follows that for every $i \in [t]$ either $V_i \subseteq A \cup S$ or $V_i \subseteq B \cup S$. This defines some partition $I_1 \cup I_2$ of $[t]$, where $I_1 = \{i \in [t] : V_i \subseteq A \cup S\}$ and $I_2 = \{i \in [t] : V_i \subseteq B \cup S\}$. We claim that this partition is non-trivial. Indeed, without loss of generality, suppose for a contradiction that $I_1 = [t]$ and $I_2 = \emptyset$. Let $u \in B$ be the first vertex according to the degeneracy ordering of $W$; such a vertex exists since $B \neq \emptyset$ and $B \subseteq W$. Since $E_G(A, B) = \emptyset$, it follows that $\deg_G(u, S) \geq 3s/4 > |S|$, which is clearly a contradiction. For every vertex $w_i \in A \cap W$, it holds that $\deg_G(w_i, B) = 0$ and thus $\deg_G(w_i, A \cap (V_1 \cup \ldots \cup V_t \cup \{w_1, \ldots, w_{i-1}\})) \geq 3s/4 - k \geq 3s/8$, implying that $w_i \in A_r$ (here $A_r$ is the set defined above, corresponding to $I_1$). Similarly, $w_i \in B \cap W$ implies $w_i \in B_r$. Therefore, the partition $A_r \cup B_r$ of $V$ induces a partition $S_1 \cup S_2$ of $S$ such that $A \cup S_1 = A_r$ and $B \cup S_2 = B_r$. Since $R$ satisfies ${\mathcal M}_k$ by assumption, and $|S| < k$, it follows that $E_{G \cup R}(A, B) \neq \emptyset$, a contradiction.

Next, we prove (b). Let $G_0$ be the disjoint union of $r := \lfloor n/(s+1) \rfloor$ cliques, $Q_1, \ldots, Q_r$, where $s+1 \leq |V(Q_i)| \leq (1 + o(1)) s$ for every $i \in [r]$; then $\delta(G_0) \geq s$. Let $H$ be some $n$-vertex graph such that $e(H) < k \lfloor n/(s+1) \rfloor /2$. Then $e_H(V_i, V(G_0) \setminus V_i) < k$ holds for some $i \in [r]$, implying that $G_0 \cup H$ is not $k$-connected; this concludes the proof of the first part of the statement. As for its second part, let $R \sim \mathbb{G}(n,p)$, where $p := p(n) \leq \frac{\tilde{c} \log(n/s)}{n s}$ for some sufficiently small constant $\tilde{c} > 0$. Consider the following auxiliary random graph $H$: its vertex-set is $\{u_1, \ldots, u_r\}$ and for every $1 \leq i < j \leq r$ there is an edge of $H$ connecting $u_i$ and $u_j$ if and only if $E_R \left(V(Q_i), V(Q_j) \right) \neq \emptyset$. Observe that $|V(H)| = r$ tends to infinity with $n$, as $s = o(n)$ by assumption. For every $1 \leq i < j \leq r$ it holds that 
\begin{align*}
\mathbb{P} \left(E_R \left(V(Q_i), V(Q_j) \right) \neq \emptyset \right) &= 1 - (1-p)^{\left|V(Q_i) \right| \left|V(Q_j) \right|} \leq 1 - (1-p)^{(1 + o(1)) s^2} \leq 1 - e^{- (1 + o(1)) p s^2} \\ 
&\leq (1 + o(1)) p s^2 \leq \frac{(1 + o(1)) \log (n/s)}{n/s} \leq \frac{(1 - o(1)) \log r}{r} \,.
\end{align*}
It follows that a.a.s. $u_i$ is isolated in $H$ for some $i \in [r]$, implying that $G_0 \cup R$ is a.a.s. disconnected.
\end{proofof}

\section{The diameter of randomly perturbed graphs} \label{sec::diameter}

In this section we prove Theorem~\ref{th::Diameter}. We begin by stating two results that facilitate our proof. The first is a standard result in the theory of random graphs. 

\begin{theorem} [Theorem 7.2 in~\cite{FK}] \label{th::DiameterGnp}
If $p = \omega(\log n/n)$, then a.a.s. $\emph{diam}(\mathbb{G}(n,p)) = \frac{(1 + o(1)) \log n}{\log(n p)}$.
\end{theorem}

\begin{proposition} \label{prop::DiametertSplit}
Let $G$ be an $n$-vertex graph with minimum degree $k := k(n) \geq 1$. Then there exists a partition $U_1 \cup \ldots \cup U_s$ of $V(G)$ such that, for every $i \in [s]$, the radius of $G[U_i]$ is at most 2 and $|U_i| \geq k+1$.  
\end{proposition} 

\begin{proof}
Let $H_1, \ldots, H_s$ be an inclusion maximal collection of pairwise vertex-disjoint stars in $G$, each having at least $k$ edges. For every $i \in [s]$ let $V_i = V(H_i)$ and let $v_i$ denote the centre of $H_i$. Let $H = G \setminus (H_1 \cup \ldots \cup H_s)$. It follows by the maximality of the family $H_1, \ldots, H_s$ that $\deg_H(u) < k \leq \deg_G(u)$ holds for every $u \in V(H)$. Hence, for every $u \in V(H)$ there exists some $i \in [s]$ such that $N_G(u) \cap V_i \neq \emptyset$; add $u$ to $V_i$ (if there is more than one such set $V_i$, then choose one arbitrarily). For every $i \in [s]$ denote the extension of $V_i$ thus obtained by $U_i$. It is then evident that $U_1 \cup \ldots \cup U_s$ is a partition of $V(G)$ and that $|U_i| \geq |V_i| \geq k+1$ for every $i \in [s]$. Moreover, for every $i \in [s]$, it follows by construction that $\textrm{dist}_{G[U_i]}(u, v_i) \leq 2$, implying that the radius of $G[U_i]$ is at most 2.   
\end{proof}  

\begin{proofof}{Theorem~\ref{th::Diameter}}
Starting with (a), let $R \sim \mathbb{G}(n,p)$, where $p := p(n) = \omega \left(\frac{\log (n/k)}{n k} \right)$. Let $U_1 \cup \ldots \cup U_s$ be a partition of $V(G)$ as in the statement of Proposition~\ref{prop::DiametertSplit}. For every $i \in [s]$ let $A^i_1 \cup \ldots \cup A^i_{m_i}$ be an arbitrary partition of $U_i$ such that $k \leq |A^i_j| \leq 2k$ for every $j \in [m_i]$; such a partition exists since $|U_i| \geq k+1$ holds by Proposition~\ref{prop::DiametertSplit}. Observe that, given two vertices $x, y \in A^i_j$, it may hold that $\textrm{dist}_{G[A^i_j]}(x,y) = \infty$, but $\textrm{dist}_G(x,y) \leq \textrm{dist}_{G[U_i]}(x,y) \leq 4$ holds by Proposition~\ref{prop::DiametertSplit}. Let $\mathcal{S} = \{(i,j) : i \in [s], j \in [m_i]\}$ and let $m = |\mathcal{S}|$; note that $n/(2k) \leq m \leq n/k$. Consider the following auxiliary random graph $H$: its vertex set is $\mathcal{S}$, and for every $(i_1, j_1), (i_2, j_2) \in \mathcal{S}$ there is an edge of $H$ connecting $(i_1, j_1)$ and $(i_2, j_2)$ if and only if $E_R \left(A^{i_1}_{j_1}, A^{i_2}_{j_2} \right) \neq \emptyset$. Observe that $|V(H)| = m$ tends to infinity with $n$ owing to $k = o(n)$ assumed in the premise. Additionally, 
\begin{align*}
\mathbb{P} \left(E_R \left(A^{i_1}_{j_1}, A^{i_2}_{j_2} \right) \neq \emptyset \right) &= 1 - (1-p)^{\left|A^{i_1}_{j_1} \right| \left|A^{i_2}_{j_2} \right|} \geq 1 - (1-p)^{k^2} \geq 1 - e^{- p k^2}\\ 
&\geq \min \{1/2, p k^2/2\} = \omega \left(\frac{\log (n/k)}{n/k} \right) = \omega \left(\frac{\log m}{m} \right)
\end{align*}
holds for every $(i_1, j_1), (i_2, j_2) \in \mathcal{S}$, independently of all other pairs. Recalling that $q = 1 - (1-p)^{k^2}$, it thus follows by Theorem~\ref{th::DiameterGnp} that a.a.s.
\begin{equation} \label{eq::auxDiam}
\textrm{diam}(H) \leq \textrm{diam}(\mathbb{G}(m,q)) \leq \frac{(1 + o(1)) \log m}{\log(m q)} \leq \frac{(1 + o(1)) \log (n/k)}{\log(n q/(2k))}.
\end{equation}

Fix any two vertices $x, y \in V(G)$. Let $(i_1, j_1) (i_2, j_2) \ldots (i_r, j_r)$ be a shortest path in $H$, where $x \in A^{i_1}_{j_1}$ and $y \in A^{i_r}_{j_r}$. For every $\ell \in [r-1]$, let $u_{\ell} \in A^{i_{\ell}}_{j_{\ell}}$ and $v_{\ell} \in A^{i_{\ell+1}}_{j_{\ell+1}}$ be vertices for which $u_{\ell} v_{\ell} \in E(G)$ (note that it is possible that $u_{\ell+1} = v_{\ell}$ for some values of $\ell$). Then $x P_1 u_1 v_1 P_2 u_2 \ldots v_{r-1} P_r y$ is a path in $G$, where each $P_i$ is a shortest (possibly trivial) path in $G$ between the corresponding endpoints. Since each $P_i$ is of length at most 4 by Proposition~\ref{prop::DiametertSplit} and since $x, y \in V(G)$ are arbitrary vertices, it follows by~\eqref{eq::auxDiam} that a.a.s. $\textrm{diam}(G \cup R) \leq \textrm{diam}(H) + 4 (\textrm{diam}(H) + 1) \leq \frac{(5 + o(1)) \log (n/k)}{\log (n q/(2k))}$ as claimed. 


Next, we prove (b). Let $G_0$ be the disjoint union of $r := \lfloor n/(k+1) \rfloor$ cliques $Q_1, \ldots, Q_r$, where $k+1 \leq |V(Q_i)| \leq (1 + o(1)) k$ for every $i \in [r]$; then $\delta(G_0) \geq k$. Let $H$ be the auxiliary random graph with vertex-set $\{u_1, \ldots, u_r\}$ such that $u_i u_j \in E(H)$ if and only if $E_R(V(Q_i), V(Q_j)) \neq \emptyset$. Similar calculations to the ones made in Part (a) of this proof show that a.a.s.
$$
\textrm{diam}(G_0 \cup R) \geq \textrm{diam}(H) \geq \textrm{diam}(\mathbb{G}(r, (1 + o(1)) q)) \geq \frac{(1 - o(1)) \log (n/k)}{\log (n q/k)}.
$$ 
\end{proofof}



\section*{Acknowledgements}

We thank Noga Alon for helpful comments.


\begin{thebibliography}{99}



\bibitem{ADHLlarge}
E.~Aigner-Horev, O.~Danon, D.~Hefetz, and S.~Letzter, \emph{Large rainbow cliques in randomly perturbed dense graphs}, {SIAM Journal on Discrete Mathematics}, to appear.

\bibitem{ADHLsmall}
E.~Aigner-Horev, O.~Danon, D.~Hefetz, and S.~Letzter,  \emph{Small rainbow cliques in randomly perturbed dense graphs}, {European Journal of Combinatorics}~\textbf{101} (2022), 103452.

\bibitem{AHhamilton}
E.~Aigner-Horev and D.~Hefetz, \emph{Rainbow Hamilton cycles in randomly coloured randomly perturbed dense graphs}, {SIAM Journal on Discrete Mathematics} \textbf{35}  (2021), 1569--1577.

\bibitem{AHK}
E.~Aigner-Horev, D.~Hefetz, and M.~Krivelevich, \emph{Cycle lengths in randomly perturbed graphs}, {Arxiv preprint arXiv:2206.12210}, 2022.

\bibitem{AHTrees}
E.~Aigner-Horev, D.~Hefetz, and A.~Lahiri, \emph{Rainbow trees in uniformly edge-coloured graphs}, {Random Structures \& Algorithms}, to appear.

\bibitem{AHP}
E.~Aigner-Horev and Y. Person, \emph{{Monochromatic Schur triples in randomly perturbed dense sets of integers}}, {SIAM Journal on Discrete Mathematics} \textbf{33} (2019), 2175--2180.

\bibitem{AKS79}
M.~Ajtai, J.~Koml\'os, and E. Szemer\'edi, {\em Topological complete subgraphs in random graphs}, {Studia Scientiarum Mathematicarum Hungarica} {\bf 14} (1979), 293--297.  

\bibitem{AKS81}
M.~Ajtai, J.~Koml\'os, and E. Szemer\'edi, {\em The longest path in a random graph}, Combinatorica~{\bf 1} (1981), 1--12.

\bibitem{AKS22}
N. Alon, M. Krivelevich, and B. Sudakov, {\em Complete minors and average degree -- a short proof}, Journal of Graph Theory, to appear, Arxiv preprint arXiv:2202.08530.  


\bibitem{AlonSpencer}
N.~Alon and J.~H.~Spencer, {\bf The probabilistic method}, Wiley Series in Discrete Mathematics and Optimization, Fourth edition, John Wiley \& Sons, Inc., Hoboken, NJ, 2016. 



\bibitem{BLW11}
J.~Balogh, J.~Lenz, and H.~Wu, {\em Complete minors, independent sets, and chordal graphs}, Discussiones Mathematicae Graph Theory, {\bf 31} (2011), 639--674. 

\bibitem{BK11}
J.~Balogh and A.~Kostochka, {\em Large minors in graphs with given independence number}, Discrete Mathematics {\bf 311} (2011), 2203--2215.

\bibitem{BTW17}
J.~Balogh, A.~Treglown, and A.~Z. Wagner, \emph{Tilings in randomly perturbed dense graphs}, {Combinatorics, Probability and Computing} \textbf{28} (2019), 159--176.

\bibitem{BHKM18}
W.~Bedenknecht, J.~Han, Y.~Kohayakawa, and G.~O. Mota, \emph{Powers of tight {H}amilton cycles in randomly perturbed hypergraphs}, {Random Structures \& Algorithms} \textbf{55} (2019), 795--807.


\bibitem{BFM03}
T.~Bohman, A.~Frieze, and R.~Martin, \emph{How many random edges make a dense graph {H}amiltonian?}, Random Structures \& Algorithms \textbf{22} (2003), 33--42.

\bibitem{BFKM04}
T.~{Bohman}, A.~{Frieze}, M.~{Krivelevich}, and R.~{Martin}, \emph{{Adding random edges to dense graphs}}, {Random Structures \& Algorithms} \textbf{24} (2004), 105--117.


\bibitem{BC81}
B.~Bollob\'{a}s and P.~A. Catlin, \emph{Topological cliques of random graphs}, {Journal of Combinatorial Theory Series B} \textbf{30} (1981), 224--227.


\bibitem{BCE80}
B.~Bollob\'{a}s, P.~A. Catlin, and P.~Erd\H{o}s, \emph{Hadwiger's conjecture is true for almost every graph}, {European Journal of Combinatorics} \textbf{1} (1980), 195--199.

\bibitem{BHKMPP18}
J.~B\"ottcher, J.~Han, Y.~Kohayakawa, R.~Montgomery, O.~Parczyk, and Y.~Person, \emph{Universality for bounded degree spanning trees in randomly perturbed graphs}, {Random Structures \& Algorithms} \textbf{55} (2019), 854--864.

\bibitem{BMPP18}
J.~B{\"o}ttcher, R. Montgomery, O.~Parczyk, and Y.~Person, \emph{Embedding spanning bounded degree graphs in randomly perturbed graphs}, Mathematika \textbf{66} (2020), 422--447.



\bibitem{DKM21}
S.~Das, C.~Knierim, and P.~Morris, \emph{Schur's theorem for randomly perturbed sets}, Extended abstracts {E}uro{C}omb 2021 (J.~Ne\v{s}et\v{r}il, G.~Perarnau, J.~Ru\'e, and O.~Serra, eds.), Trends in {M}athematics, vol.~14, Birkh\"auser, 2021.

\bibitem{DMT20}
S.~Das, P.~Morris, and A.~Treglown, \emph{{Vertex Ramsey properties of randomly perturbed graphs}}, {Random Structures \& Algorithms} \textbf{57} (2020), 983--1006.



\bibitem{DT19}
S.~Das and A.~Treglown, \emph{Ramsey properties of randomly perturbed graphs: cliques and cycles}, {Combinatorics, Probability and Computing} \textbf{29} (2020), 830--867.


\bibitem{DP22}
M. Delcourt and L. Postle, {\em Reducing linear Hadwiger’s conjecture to coloring small graphs}, {Arxiv preprint  	arXiv:2108.01633}, 2022. 



\bibitem{DM82}
P.~Duchet and H.~Meyniel, {\em On Hadwiger’s number and the stability number}, {Annals of Discrete Mathematics}~\textbf{13} (1982), 71--73.

\bibitem{DRRS18}
A.~Dudek, C.~Reiher, A.~Ruci{\'n}ski, and M.~Schacht, \emph{Powers of {H}amiltonian cycles in randomly augmented graphs}, {Random Structures \& Algorithms} \textbf{56} (2020), 122--141.



\bibitem{EF81}
P.~Erd\H{o}s and S.~Fajtlowicz, {\em On the conjecture of Haj\'os}, {Combinatorica}~{\bf 1} (1981), 141--143. 


\bibitem{DiazGeo}
A.~Espuny~D\'iaz, {\em Hamiltonicity of graphs perturbed by a random geometric graph}, {Arxiv preprint arXiv:2102.02321}, 2021.   

\bibitem{DiazPower}
A.~Espuny~D\'iaz and J.~Hyde, {\em Powers of Hamilton cycles in dense graphs perturbed by a random geometric graph
}, {Arxiv preprint arXiv:2205.08971}, 2022.   



\bibitem{DiazReg}
A.~Espuny~D\'iaz and A.~Gir\~ao, {\em Hamiltonicity of graphs perturbed by a random regular graph}, {Arxiv preprint arXiv:2102.02321}, 2021.   


\bibitem{FKS08}
N.~Fountoulakis, D.~K\"{u}hn, and D.~Osthus,  \emph{The order of the largest complete minor in a random graph}, {Random Structures \& Algorithms}~\textbf{33} (2008), 127--141.

\bibitem{FKS09}
N.~Fountoulakis, D.~K\"{u}hn, and D.~Osthus, \emph{Minors in random regular graphs}, {Random Structures \&
  Algorithms} \textbf{35} (2009), 444--463.


\bibitem{Fox}
J.~Fox, {\em Complete minors and independence number}, SIAM Journal on Discrete Mathematics {\bf 24} (2010), 1313--1321.

\bibitem{FK}
A.~Frieze and M.~Karo\'nski, {\bf Introduction to Random Graphs}, {Cambridge University Press}, 2015.


\bibitem{Hadwiger}
H.~Hadwiger, {\em \"Uber eine Klassifikation der Streckenkomplexe}, Vierteljschr. Naturforsch. Ges.
Z\"urich {\bf 88} (1943), 133--142.

\bibitem{HZ18}
J.~Han and Y.~Zhao, \emph{Hamiltonicity in randomly perturbed hypergraphs}, {Journal of Combinatorial Theory Series B} \textbf{144} (2020), 14--31.


\bibitem{JLR}
S.~Janson, T.~\L uczak, and A.~Ruci\'nski, \textbf{Random graphs}, Wiley-Interscience Series in Discrete
Mathematics and Optimization, Wiley-Interscience, New York, 2000.


\bibitem{KKKO20}
D.~Y.~Kang, M.~Kang, J.~Kim, and S.~Oum, {\em Fragile minor-monotone parameters under random edge perturbation}, Arxiv preprint arXiv:2005.09897, 2020.  

\bibitem{KPT05}
K.~Kawarabayashi, M.~Plummer, and B.~Toft, {\em Improvements of the theorem of Duchet and Meyniel on Hadwiger’s conjecture}, Journal of Combinatorial Theory Series B {\bf 95} (2005), 152--167.

\bibitem{KS07}
K.~Kawarabayashi and Z.~X.~Song, {\em Independence number and clique minors}, Journal of Graph Theory {\bf 56}
(2007), 219--226.

\bibitem{Kos84}
A.~V.~Kostochka, \emph{Lower bound of the {H}adwiger number of graphs by their average degree}, Combinatorica \textbf{4} (1984), 307--316.


\bibitem{K18}
M.~Krivelevich, \emph{Finding and using expanders in locally sparse graphs}, {SIAM Journal on Discrete Mathematics}~\textbf{32} (2018), 611--623.

\bibitem{KKS16}
M.~Krivelevich, M.~Kwan, and B.~Sudakov, \emph{Cycles and matchings in randomly perturbed digraphs and hypergraphs}, {Combinatorics, Probability and Computing} \textbf{25} (2016), 909--927.

\bibitem{KKS17}
M.~Krivelevich, M.~Kwan, and B.~Sudakov, \emph{Bounded-degree spanning trees in randomly perturbed graphs}, {SIAM Journal on Discrete Mathematics} \textbf{31} (2017), 155--171.


\bibitem{KN21}
M.~Krivelevich and R.~Nenadov, {\em Complete minors in graphs without sparse cuts}, International Mathematics Research Notices {\bf 12} (2021), 8996--9015.



\bibitem{KS13}
M.~Krivelevich and B.~Sudakov, {\em The phase transition in random graphs: a simple proof}, Random Structures \& Algorithms~{\bf 43} (2013), 131--138.

\bibitem{KST}
M.~Krivelevich, B.~Sudakov, and P.~Tetali, \emph{On smoothed analysis in dense graphs and formulas}, Random Structures \& Algorithms \textbf{29} (2006), 180--193.




  
\bibitem{Mader}
W. Mader, Existenz n-fach zusammenh\"angender Teilgraphen in Graphen
gen\"ugend grosser Kantendichte, \emph{Abh. Math. Sem. Univ. Hamburg} 37 (1972), 86--97.  

\bibitem{MM87}
F.~Maffray and H.~Meyniel, {\em On a relationship between Hadwiger and stability numbers}, {Discrete
Mathematics} {\bf 64} (1987), 39--42.

\bibitem{MM18}
A.~McDowell and R.~Mycroft, \emph{Hamilton {$\ell$}-cycles in randomly perturbed hypergraphs}, {Electronic Journal of Combinatorics} \textbf{25} (2018), Paper 4.36.


\bibitem{NPS19}
S.~Norin, L.~Postle, and Z.~Song, \emph{Breaking the degeneracy barrier for coloring graphs with no $K_t$ minor}, Arxiv preprint arXiv:1910.09378, 2019.


\bibitem{PT10}
A.~S.~Pedersen and B.~Toft, {\em A basic elementary extension of the Duchet-Meyniel theorem}, Discrete Mathematics {\bf 310} (2010), 480--488. 

\bibitem{PST03}
M.~D.~Plummer, M.~Stiebitz, and B.~Toft, {\em On a special case of Hadwiger’s Conjecture}, Discussiones Mathematicae Graph Theory {\bf 23} (2003), 333--363.

\bibitem{Postle}
L.~Postle, \emph{{Further progress towards Hadwiger’s conjecture}}, Arxiv preprint arXiv:2006.11798, 2020.

\bibitem{Powierski19}
E.~Powierski, \emph{Ramsey properties of randomly perturbed dense graphs}, Arxiv preprint arXiv:1902.02197, 2019.


\bibitem{SUbook}
E.~R.~Scheinerman and D.~H.~Ullman, \textbf{Fractional graph theory}, John Wiley \& Sons, Inc., New York, 1997.

\bibitem{Tho95}
A.~Thomason, \emph{An extremal function for contractions of graphs}, Mathematical Proceedings of the Cambridge Philosophical Society~\textbf{95}
 (1984), 261--265.


\bibitem{Wood}
D. R. Wood, {\em Independent sets in graphs with an excluded clique minor}, Discrete Mathematics \& Theoretical Computer Science {\bf 9} (2007), 171--175.

\bibitem{Woodall}
D.~R.~Woodall, {\em Subcontraction-equivalence and Hadwiger’s Conjecture}, Journal of Graph Theory~{\bf 11}
(1987), 197--204.

\end{thebibliography}
\end{document}